\documentclass{amsart}
\usepackage{amsmath,amsfonts,amssymb,amscd}
\usepackage[francais]{babel}

\newtheorem{theoreme}[equation]{Th\'eor\`eme}
\newtheorem{lemme}[equation]{Lemme}
\newtheorem{lemme-definition}[equation]{Lemme-D\'efinition}
\newtheorem{proposition-definition}[equation]{Proposition-D\'efinition}
\newtheorem{proposition}[equation]{Proposition}
\newtheorem{definition}[equation]{D\'efinition}
\newtheorem{corollaire}[equation]{Corollaire}

\theoremstyle{remark}
\newtheorem{remarque}[equation]{Remarque}

\newtheorem{notation}[equation]{Notation}
\numberwithin{equation}{section}

\DeclareMathOperator{\Id}{Id}

\newcommand{\bbN}{{\mathbb N}}

\newcommand{\bbZ}{{\mathbb Z}}

\newcommand{\bI}{{\bf I}}
\newcommand{\bJ}{{\bf J}}

\newcommand{\bS}{{\bf S}}

\newcommand{\bb}{{\bf b}}

\newcommand{\bp}{{\bf p}}
\newcommand{\br}{{\bf r}}
\newcommand{\bs}{{\bf s}}
\newcommand{\bt}{{\bf t}}
\newcommand{\bu}{{\bf u}}
\newcommand{\bv}{{\bf v}}
\newcommand{\bw}{{\bf w}}
\newcommand{\bx}{{\bf x}}
\newcommand{\by}{{\bf y}}
\newcommand{\bz}{{\bf z}}

\newcommand{\BW}{{B^+}}

\newcommand{\inv}{^{-1}}

\newcommand{\cf}{{\it cf.}}

\newcommand{\ie}{{\it i.e.}}

\newcommand{\lexp}[2]{\kern\scriptspace\vphantom{#2}^{#1}\kern-\scriptspace#2}

\def\BW{{B^+_W}}
\def\Bred{{B_W^{\text{red}}}}
\def\wmst{{\bw_{\{\bs,\bt\}}}}
\def\wst#1{{<\bs,\bt;#1>}}
\def\wts#1{{<\bt,\bs;#1>}}

\def\barlength{10 pt}
\def\node{{\kern -1pt\circ\kern -1pt}}
\def\bulnode{{\kern -1pt\bullet\kern -1pt}}
\def\dbulnode#1{{\kern-1pt\displaystyle\mathop\bullet^{\hbox to 0pt{\hss
$\scriptstyle#1$\hss}}\kern-1pt}}

\def\dbar{{\rlap{\vrule width\barlength height2pt depth-1.5pt} 
                 \vrule width\barlength height3pt depth-2.5pt}
		 }
\def\bar{{\vrule width\barlength height2.75pt depth-2.25pt}}
\def\vertbar#1#2{\rlap{\kern 0.7pt\vrule width0.5pt height1pt depth9pt}
		 \rlap{\lower12.7pt\rlap{$#2$}}#1}

\def\cf{{\it cf. }}

\def\ie{{\it i.e.}}

\begin{document}
\title[Groupes de tresses purs]{Pr\'esentation des groupes de tresses purs
et de certaines de leurs extensions}
\author{Fran\c cois Digne}
\address{LAMFA, Universit\'e de Picardie,
33 Rue Saint-Leu, 80039 Amiens France}

\maketitle
\section{Introduction et notations}
Dans cet article on \'etudie par des m\'ethodes combinatoires
la pr\'esentation du groupe de tresses purs associ\'e \`a
un groupe de Coxeter quelconque.
On g\'en\'eralise \`a tous les types le d\'evissage 
du groupe de tresses pur en produits semi-directs successifs, 
(connu dans le cas du type $A_n$, \cf  par exemple
[\cite{Birman}, 1.8.2] ou [\cite{Hansen}, appendice
par G\ae dde]). On montre en particulier que
pour le type $B_n$ on a (comme pour $A_n$)
une suite d'extensions de groupes libres.

Dans ce qui suit $(W,S)$ est un syst\`eme de Coxeter et $\BW$ (resp. $B_W$)
est le mono{\"\i}de de tresses (resp. le groupe de tresses) correspondant.
Rappelons les d\'efinitions.
Pour $s$ et $t$ dans $S$,
notons $m_{s,t}\in N\cup\{\infty\}$ l'ordre de $st$ dans $W$ et soit
$\bS$ un ensemble en bijection avec $S$, la bijection \'etant not\'ee
$\bs\mapsto s$. Avec ces notations
$\BW$ (resp. $B_W$) est le mono\"\i de
(resp. le groupe) engendr\'e par $\bS$ avec comme seules relations
$\underbrace{\bs\bt\ldots}_{m_{s,t}}=
\underbrace{\bt\bs\ldots}_{m_{s,t}}$, pour tout $\bs\in \bS$ et tout 
$\bt\in\bS$.

On a donc un morphisme canonique $p$ de $B_W$ dans $W$ qui envoie
$\bs$ sur $s$. On note $P_W$ le
noyau de $p$ (groupe de tresses pur).
On a aussi un homomorphisme canonique $j :\BW\to B_W$ (dont on ne sait
pas en g\'en\'eral s'il est injectif).
La surjection $p\circ j$ a une section canonique obtenue
en relevant les \'ecritures r\'eduites des \'el\'ements de $W$ (c'est
ind\'ependant de l'\'ecriture r\'eduite).
On d\'esigne par
$\bw$ le relev\'e dans $\BW$ de $w\in W$ par cette section,
ainsi que son image dans $B_W$.
On note $\Bred$ l'ensemble des relev\'es par la section des
\'el\'ements de $W$, ou l'ensemble de leurs images dans
$B_W$ (ces deux ensembles
sont mis en bijection par $j$).
L'ensemble $\Bred$ est constitu\'e
des \'el\'ements $\bs_1\bs_2\ldots\bs_k$ dont
l'image dans $W$ est de longueur $k$.

Si $\bb$ est un \'el\'ement de $B_W$ de la forme
$\bb=\bs_1^{\varepsilon_1}\bs_2^{\varepsilon_2}\ldots\bs_k^{\varepsilon_k}$,
avec $\varepsilon_i=\pm1$,
on pose $\tilde\bb=\bs_k^{\varepsilon_k}\ldots\bs_2^{\varepsilon_2}
\bs_1^{\varepsilon_1}$. Cet \'el\'ement est bien d\'efini, ind\'ependamment de la
d\'ecomposition de $\bb$ en produit de g\'en\'erateurs car les relations de
d\'efinition de $B_W$ sont invariantes par retournement des mots.
Si $I$ est une partie de $S$ on note
$W_I$ le sous-groupe (parabolique) de
$W$ engendr\'e par $I$.
Si $I$ est un partie sph\'erique de $S$ (c'est-\`a-dire si $W_I$ est
fini), on note $w_I$ l'\'el\'ement
de plus grande
longueur du sous-groupe parabolique $W_I$ engendr\'e par $I$;
on notera $\bI$ la partie de $\bS$ correspondant \`a $I$ et $\bw_\bI$ le relev\'e
dans $\Bred$ de $w_I$.
Nous serons amen\'es \`a utiliser souvent le lemme d'\'echange, ou plut\^ot une
version adapt\'ee au mono\"\i de de tresses qui est
\begin{lemme}\label{echange} Soient $\bs$ et $\bt$ dans $\bS$ et
soit $\bb\in\Bred$;
si $\bs\bb\in\Bred$ et $\bb\bt\in\Bred$ mais
$\bs\bb\bt\not\in\Bred$ alors $\bs\bb=\bb\bt$.
\end{lemme}
\section{La fonction $N$}
Soit $T$ l'ensemble des r\'eflexions de $W$, c'est-\`a-dire des conjugu\'es des
\'el\'ements de $S$.
\begin{proposition} \label{N(b)}
\begin{enumerate}
\item Il existe une application $N:B_W\to\bbZ T$ telle que
$$N(\bs_1^{\varepsilon_1}\bs_2^{\varepsilon_2}\ldots\bs_k^{\varepsilon_k})=
\sum_{i=1}^{i=k}\varepsilon_i s_1s_2\ldots s_{i-1}s_is_{i-1}\ldots
s_1,$$
o\`u $\bs_i\in\bS$ et $\varepsilon_i=\pm1$.
\item L'application $\bb\mapsto(N(\bb),p(\bb))$ est un homomorphisme de
groupes de $B_W$ dans $\bbZ T\rtimes W$, o\`u l'action de $W$ sur $\bbZ T$ \'etend
lin\'eairement l'action de $W$ sur $T$ par conjugaison.
\end{enumerate}
\end{proposition}
\begin{proof}
Nous allons montrer qu'il existe un homomorphisme de groupes $f:B_W\to
\bbZ T\rtimes W$ qui envoie $\bs\in\bS$ sur $(s,s)$. 
La composition de $f$ avec la projection sur la premi\`ere composante
donne alors une application $N$ qui v\'erifie (i).
Pour montrer l'existence de $f$ il suffit de v\'erifier les relations de tresses
c'est-\`a-dire de montrer que pour tout couple $(\bs,\bt)$ d'\'el\'ements de $\bS$
on a la relation
$f(\bs)f(\bt)\ldots=f(\bt)f(\bs)\ldots$, o\`u il y a $m_{s,t}$ facteurs
dans les deux membres. Or dans le produit semi-direct $\bbZ T\rtimes W$
la premi\`ere composante de
$(s,s)(t,t)\ldots$ (avec $m_{s,t}$ facteurs) est la somme de toutes les
r\'eflexions du groupe di\'edral engendr\'e par $s$ et $t$ et la deuxi\`eme
composante est l'\'el\'ement de plus grande longueur de ce groupe
di\'edral. On trouve donc le m\^eme r\'esultat quand on intervertit $\bs$ et $\bt$,
d'o\`u la relation et la proposition.
\end{proof}
\begin{lemme}\label{restriction a P_W}
Le restriction de $N$ \`a $P_W$ est un homomorphisme compatible avec les
actions de
$B_W$ respectivement par conjugaison sur $P_W$ et par conjugaison \`a travers $p$
sur $T$.
\end{lemme}
\begin{proof} La restriction de $(N,p)$ \`a $P_W$ est simplement la
restriction de $N$, d'o\`u les deux assertions.
\end{proof}
\begin{lemme}\label{image de P_W}
L'image de $P_W$ par $N$ est $2\bbZ T$.
\end{lemme}
\begin{proof}
Le groupe $P_W$ est engendr\'e par les conjugu\'es
des carr\'es des r\'eflexions par d\'efinition, donc son image par $(N,p)$
est incluse dans le
plus petit sous-groupe normal contenant les images des carr\'es des
r\'eflexions qui sont les
$(2s,1)$ et qui sont dans le sous-groupe normal $2T\times\{1\}$ de
$\bbZ T\rtimes W$. Dans l'autre sens,
pour tout $\bs\in\bS$ et tout $n\in\bbZ$
l'image de $\bs^{2n}\in P_W$ est $2ns$.
Comme l'image $N(P_W)$ est stable par action de $W$, elle contient donc
$2\bbZ T$, d'o\`u l'\'egalit\'e cherch\'ee.
\end{proof}
Par passage au quotient par $P_W$ on obtient un homomorphisme
$(\overline N,\Id):W\to \bbZ/2\bbZ T\rtimes W$.
Il est connu que l'application $\overline N$ est
injective (\cf [\cite{Dyer}, chapitre 1], par exemple).
On sait aussi que l'image de $\overline N$, consid\'er\'ee comme un ensemble de
parties de $T$ est caract\'eris\'ee par le fait d'\^etre admissible au sens
suivant (\cf [\cite{Papi}],
o\`u cela est fait dans un cadre un peu plus
restrictif mais facile \`a g\'en\'eraliser):
\begin{definition}\label{admissible}
Une partie finie de $T$ est dite admissible si
elle correspond \`a un ensemble de racines positives clos par combinaison
lin\'eaire \`a coefficients r\'eels positifs et
dont le compl\'ementaire dans les racines positives est aussi clos.
\end{definition}
On a alors
\begin{proposition} \label{image de N}
L'image de $(N,p)$ est l'ensemble des $(x,w)$ tel que l'image de $x$ modulo 2
soit $\overline N(w)$. L'image de $N$ est l'ensemble des \'el\'ements
de $\bbZ T$ dont l'image modulo 2 est admissible.
\end{proposition}
\begin{proof} La deuxi\`eme assertion est cons\'equence de la
premi\`ere. La propri\'et\'e rappel\'ee plus haut de l'image de $\overline N$
montre que l'image de $(N,p)$ est incluse dans l'ensemble de l'\'enonc\'e.
Montrons l'inclusion en sens inverse.
Si l'image $\overline A$ modulo 2 de
$A\in\bbZ T$ est \'egale \`a $\overline N(w)$ pour un certain \'el\'ement $w\in W$,
relevons $w$ par $\bw\in\Bred$; on
a $N(\bw)-A\in 2\bbZ T$. Donc par le lemme \ref{image de P_W}
il existe $\bp\in P_W$
tel que $N(\bp)=N(\bw)-A$, ce qui implique $N(\bp\inv\bw)=A$
par \ref{N(b)} (ii).\end{proof}

\begin{proposition} \label{noyau de N}
Le noyau de $(N,p)$ est le d\'eriv\'e du groupe de tresses pur
et la restriction de $N$ \`a $P_W$ est l'homomorphisme canonique
de $P_W$ dans son ab\'elianis\'e.
\end{proposition}
\begin{proof}
Le noyau de $(N,p)$ est inclus dans $P_W$.
La restriction de $(N,p)$ \`a $P_W$ qui est la restriction de $N$
est un homomorphisme $P_W\rightarrow2\bbZ T$.
Le noyau contient donc le d\'eriv\'e de $P_W$ et $N$ se factorise donc par
l'ab\'elianis\'e de $P_W$. Pour obtenir le r\'esultat il suffit de trouver un ensemble
de g\'en\'erateurs de $P_W$ qui s'envoie bijectivement sur une base de $2\bbZ T$. 
La proposition suivante nous fournit un tel ensemble, d'o\`u le r\'esultat.
\end{proof}
\begin{proposition} \label{generateurs de P_W}
Pour tout choix d'un ensemble de couples $(\bw,\bs)\in\Bred\times\bS$ tel
que les \'el\'ements
$\bw\bs\tilde\bw\in\Bred$ parcourent l'ensemble des relev\'es
des r\'eflexions de $W$ 
les \'el\'ements $\bw\bs^2\bw\inv$ engendrent $P_W$.
\end{proposition}
On retrouve ainsi le r\'esultat connu pour les groupes de tresses de
type fini (y compris ceux associ\'es \`a des groupes de r\'eflexions complexes, \cf
[\cite{BrMaRo}, 2.2])
que l'ab\'elianis\'e de $P_W$ est le groupe ab\'elien libre engendr\'e par
les r\'eflexions de $W$.
\begin{proof}
Appliquons la m\'ethode de Reidemeister-Schreier ([\cite{Johnson}, chapitre 9])
au  sous-groupe $P_W$ de $B_W$. Le quotient \'etant $W$, on peut prendre
comme repr\'esentants des classes modulo $P_W$ les \'el\'ements de $\Bred$.
On en d\'eduit que $P_W$ est engendr\'e par les $\bv\bs q(vs)\inv$, o\`u $\bv$
d\'ecrit $\Bred$, $\bs$ d\'ecrit $\bS$ et o\`u $q$ est la
section de $p$. Or si $l(vs)<l(v)$, on peut \'ecrire $\bv=\bw\bs$ avec
$\bw\in\Bred$ et l'on a $\bv\bs q(vs)\inv=\bw\bs^2\bw\inv$. Si $l(vs)>l(v)$,
on a $\bv\bs q(vs)\inv=1$. Donc $P_W$ est engendr\'e par les
$\bw\bs^2\bw\inv$ avec $l(ws)>l(w)$.
Nous devons voir qu'on peut se limiter aux $\bw$ et $\bs$ tels que
$\bw\bs\tilde\bw$ soit dans $\Bred$ et parcoure un ensemble de relev\'es des
r\'eflexions de $W$. C'est l'objet des deux lemmes suivants.
\begin{lemme}\label{un seul generateur par reflexion}
Si $\bt$ et $\bt'$ sont deux \'el\'ements de $\bS$ et $\bw$ et $\bw'$ deux
\'el\'ements de $\Bred$ tels que
$\bw\bt\tilde\bw=\bw'\bt'\tilde{\bw'}\in\Bred$, alors $\bw\bt\bw\inv$ et
$\bw'\bt'\bw^{\prime-1}$ sont conjugu\'es par un produit d'\'el\'ements de la
forme $\bv\bs^2\bv\inv$ avec $\bs\in\bS$, $\bv\bs\in\Bred$ et
$l(\bv)<l(\bw)$.
\end{lemme}
\begin{lemme}\label{ws tilde w reduit}
Si $\bw\in\Bred$ et $\bs\in\bS$ sont tels que $\bw\bs\in\Bred$ et
$\bw\bs\tilde\bw\not\in\Bred$, alors $\bw\bs^2\bw\inv$ est produit
d'\'el\'ements de la forme
$\bv\bt^{\pm2}\bv\inv$
avec $\bt\in\bS$, $l(\bv)<l(\bw)$ et $\bv\bt\tilde\bv\in\Bred$.
\end{lemme}
Avant de d\'emontrer les lemmes, rappelons le r\'esultat de Dyer:
\begin{lemme} \label{Dyer} [\cite{Dyer}, 1.4]
Si $s_1\ldots s_{2n+1}$ est une d\'ecomposition r\'eduite
d'une r\'eflexion de $W$ alors cette r\'eflexion s'\'ecrit aussi $s_1\ldots s_n
s_{n+1}s_n\ldots s_1$.
\end{lemme}
\begin{proof}[D\'emonstration du lemme \ref{un seul generateur par reflexion}]
Puisque $\bw\bt\tilde\bw=\bw'\bt'\tilde{\bw'}$ dans le mono\"\i de
$\BW$, on passe d'une \'ecriture r\'eduite  correspondant \`a la
d\'ecomposition en produit comme dans le membre de gauche, c'est-\`a-dire
form\'ee par concat\'enation
d'une \'ecriture r\'eduite de $\bw$, de $\bt$ et d'une \'ecriture r\'eduite de
$\tilde{\bw}$ \`a une \'ecriture  r\'eduite analogue correspondant au membre
de droite par une suite de relations de
tresses. Pr\'ecis\'ement,
il existe une suite de triplets que nous supposerons non redondante
$(\bv_i,\bs_i,\bw_i), i=1\ldots n$
tels que $\bv_i$ et $\bw_i$ soient dans $\Bred$ et de m\^eme longueur, que
$\bs_i\in\bS$, que $(\bv_1,\bs_1,\bw_1)=(\bw,\bt,\tilde\bw)$,
que $(\bv_n,\bs_n,\bw_n)=(\bw',\bt',\tilde{\bw'})$, que
$\bv_i\bs_i\bw_i=\bv\bt\tilde\bw$ pour tout $i$ et que l'on passe de
chaque triplet au suivant en appliquant une seule relation de tresse \`a une
\'ecriture de $\bv_i\bs_i\bw_i$ obtenue par concat\'enation
d'une \'ecriture de
chacun des trois termes puis par red\'ecoupage en trois termes de
l'\'ecriture obtenue. L'image dans $W$ de tous
les produits $\bv_i\bs_i\bw_i$ est la m\^eme
r\'eflexion et par \ref{Dyer} on voit que
$\bw_i=\tilde\bv_i$.
On passe de $\bv_{i-1}\bs_{i-1}\tilde\bv_{i-1}$ \`a
$\bv_i\bs_i\tilde\bv_i$ en appliquant une relation de tresses qui fait
intervenir le facteur $\bs_{i-1}$ car on a suppos\'e la suite de triplets
non redondante.
C'est donc la relation entre $\bs$ et
$\bs_{i-1}$ pour un certain $\bs\in\bS$.
Comme $\bv_{i-1}\bs_{i-1}\tilde\bv_{i-1}$ est dans $\Bred$ l'\'el\'ement
$\bv_{i-1}$
ne peut pas \^etre divisible \`a droite par un \'el\'ement de longueur plus
grande que $\left[\frac{ m_{s,s_{i-1}}-1}
2\right]$ du mono\"\i de engendr\'e par $\bs$ et
$\bs_{i-1}$.
Donc $m_{s,s_{i-1}}$ est impair et on a $\bs_i=\bs$.
Consid\'erons alors la suite $\bv_i\bs_i\bv_i\inv$
d'\'el\'ements de $B_W$.
Pour prouver le lemme nous montrons qu'on passe d'un terme de cette
suite au terme suivant
par une suite de conjugaisons comme dans l'\'enonc\'e: un terme de la suite
est de la forme $\bx\bs\bt\ldots\bu\ldots\bt\inv\bs\inv\bx\inv$ o\`u
$\bu$ vaut $\bs$ ou $\bt$ suivant la parit\'e de $\frac
{m_{s,t}-1} 2$.
Ce terme est
\'egal \`a $\bx\bt\inv\bs\inv\ldots\bu'\ldots\bs\bt\bx\inv$,
o\`u $\bu'$ vaut $\bt$ ou $\bs$ suivant que $\bu$ vaut $\bs$ ou $\bt$.
On peut l'\'ecrire
$$\lexp{\lexp\bx\bt^{-2}.\lexp{\bx\bt}\bs^{-2}.\lexp{\bx\bt\bs}\bt^{-2}\ldots}
(\bx\bt\bs\ldots\bu'\ldots\bs\inv\bt\inv\bx\inv).$$
Ceci donne le r\'esultat puisque le terme suivant de la suite est
$\bx\bt\bs\ldots\bu'\ldots\bs\inv\bt\inv\bx\inv$.
\end{proof}
\begin{proof}[D\'emonstration du lemme \ref{ws tilde w reduit}]
On fait une r\'ecurrence sur $l(\bw)$. 
Si $\bw\bs\tilde\bw$ n'est pas dans $\Bred$,
\'ecrivons $\bw=\bs'\bw_1$. On peut appliquer l'hypoth\`ese de r\'ecurrence \`a
$\bw_1\bs^2\bw_1\inv$: cet \'el\'ement est produit d'\'el\'ements de la forme
(resp.~est de la forme) $\bv\bt^2\bv\inv$ avec
$\bv\bt\tilde\bv\in\Bred$ et $l(\bv)<l(\bw_1)<l(\bw)$ (resp.~$l(\bv)<l(\bw)$).
Donc $\bw\bs^2\bw\inv$ est produit des $\bs'\bv\bt^2\bv\inv\bs^{\prime-1}$.
Pour un tel terme, si $\bs'\bv\bt\tilde\bv\bs'$ est dans $\Bred$, on a fini
(ceci ne se produit pas si $\bv=\bw_1$).
Sinon, faisons deux cas selon que $\bs'\bv\bt$ est ou n'est pas dans $\Bred$.
Si cet \'el\'ement n'est pas dans $\Bred$,
il y a deux possibilit\'es (\cf \ref{echange}):
ou bien $\bs'\bv=\bv\bt$ et alors
$\bs'\bv\bt^2\bv\inv\bs^{\prime-1}=\bv\bt^2\bv\inv$, et on a fini;
ou bien $\bv=\bs'\bv'$ et
$\bs'\bv\bt^2\bv\inv\bs^{\prime-1}=
\bs^{\prime2}\bv'\bt^2\bv^{\prime-1}\bs^{\prime-2}$ et on finit en
appliquant l'hypoth\`ese de r\'ecurrence \`a $\bv'\bt^2\bv^{\prime-1}$.
Si $\bs'\bv\bt\in\Bred$, alors faisons de nouveau deux cas:
ou bien $\bs'\bv\bt\tilde\bv\not\in\Bred$
ou bien $\bs'\bv\bt\tilde\bv=\bv\bt\tilde\bv\bs'$ (toujours par
\ref{echange}).
\item{$\bullet$} Dans le premier cas
le lemme d'\'echange dit que $\bv\bt\tilde\bv=\bs'\bv\bt\tilde{\hat\bv}$, o\`u
$\hat\bv$ s'obtient en supprimant un g\'en\'erateur dans l'\'ecriture de $\bv$.
On peut alors appliquer \ref{Dyer} qui montre que
$\bv\bt\tilde\bv=\bs'\bu\bt'\tilde\bu \bs'$, o\`u $\bv=\bu\bt'$ avec
$\bu\in\Bred$ et $\bt'\in\bS$.
Par \ref{un seul generateur par reflexion}, on en d\'eduit que $\bv\bt\bv\inv$
est conjugu\'e de $\bs'\bu\bt'\bu\inv\bs^{\prime-1}$
par un produit d'\'el\'ements de la forme $\bx\br^2\bx\inv$ avec $\br\in\bS$,
$\bx\br\tilde\bx\in\Bred$ et $l(\bx)<l(\bv)$. On en d\'eduit que 
$\bs'\bv\bt^2\bv\inv\bs^{\prime-1}$ est conjugu\'e de
$\bs^{\prime2}\bu\bt^{\prime2}\bu\inv\bs^{\prime-2}$
par un produit d'\'el\'ements de la forme $\bs'\bx\br^2\bx\inv\bs^{\prime-1}$
avec $l(\bs'\bx)<l(\bw)$, ce qui permet de terminer ce cas par r\'ecurrence.
\item{$\bullet$} Dans le deuxi\`eme cas,
on passe de $\bs'\bv\bt\tilde\bv$ \`a $\bv\bt\tilde\bv\bs'$ par l'application
d'une suite de relations de tresses (ces \'el\'ements sont \'egaux et sont dans
$\Bred$). Consid\'erons la premi\`ere relation dans cette suite qui fait intervenir
$\bs'$: il existe $\br$ divisant \`a gauche $\bv\bt\tilde\bv$ dans $\Bred$ tel que
cette relation soit la relation entre $\bs'$ et $\br$, c'est-\`a-dire que
$\bv\bt\tilde\bv$ a une \'ecriture dans $\Bred$ qui commence
par $\bw'_{\bs',\br}$, o\`u $\bw'_{\bs',\br}$ est d\'efini par $\bw_{\bs',\br}=\bs'
\bw'_{\bs',\br}$. On a alors ou bien
\begin{enumerate}
\item[(a)]$\bv\bt\tilde\bv=\bw'_{\bs',\br}$
\end{enumerate}
ou bien
\begin{enumerate}
\item[(b)]
$\bv\bt\tilde\bv=\bw'_{\bs',\br}\bx\bu\tilde\bx\tilde\bw'_{\bs',\br}$ pour un
certain $\bx$ avec $l(\bx)<l(\bv)$, $\bu\in\bS$ et $\bs'$ ou $\br$, suivant la
parit\'e de $m_{s',r}$ commute \`a
$\bx\bu\tilde\bx$ d'apr\`es \ref{echange}.
\end{enumerate}

Dans le cas (a), posons $m_{s',r}=2k$, on a
$\bs'\bv\bt\bv\inv\bs^{\prime-1}=
\underbrace{\bs'\br\bs'\ldots}_{k+1}
\underbrace{\ldots\br\inv\bs^{\prime -1}\br\inv}_{k-1}\bs^{\prime-1}$,
ce qui, d'apr\`es la relation de tresses, vaut
$\underbrace{\br\inv\bs^{\prime-1}\br\inv\ldots}_{k-1}
\underbrace{\ldots\br\bs'\br}_k$,
et par le m\^eme argument qu'\`a la fin de la d\'emonstration
du lemme \ref{ws tilde w reduit}
cet \'el\'ement est conjugu\'e
par des \'el\'ements du type voulu \`a
$\underbrace{\br\bs'\br\ldots}_k
\underbrace{\ldots\br\inv\bs^{\prime-1}\br\inv}_{k-1}$
et on a fini (remarquer que
$\underbrace{\br\bs'\br\ldots\br\bs'\br}_{2k-1}$ est
dans $\Bred$).

Dans le cas (b), par \ref{un seul generateur par reflexion}
$\bv\bt^2\bv\inv$ est conjugu\'e par des \'el\'ements du type voulu $\bz\bs_1^2\bz\inv$
avec $l(\bz)<l(\bv)$ \`a 
$\bw'_{\bs',\br}\bx\bu^2\bx\inv(\bw'_{\bs',\br})\inv$, donc $\bs'\bv\bt\bv\inv
\bs^{\prime-1}$
est conjugu\'e par les $\bs'\bz\bs_1^2\bz\inv\bs^{\prime-1}$ (auquel on peut
appliquer l'hypoth\`ese de r\'ecurrence) \`a
$\bw_{\bs',\br}\bx\bu^2\bx\inv\bw_{\bs',\br}\inv$. Il reste \`a voir que ce dernier
\'el\'ement est bien produit d'\'el\'ements du type voulu.
Supposons que c'est $\br$ qui commute \`a $\bx\bu\tilde\bx$. On \'ecrit alors
l'\'el\'ement consid\'er\'e sous la forme $\by\br\bx\bu^2\bx\inv\br\inv\by\inv$.
On peut appliquer l'hypoth\`ese de r\'ecurrence \`a $\br\bx\bu^2\bx\inv\br\inv$
et on termine \`a nouveau par r\'ecurrence.
\end{proof}

On peut alors terminer la d\'emonstration de la proposition.
On sait que $P_W$ est engendr\'e par les $\bw\bs^2\bw\inv$ o\`u $\bw\bs\in\Bred$.
Par r\'ecurrence sur la longueur de $\bw$ on voit en utilisant les deux
lemmes
que $\bw\bs^2\bw\inv$ s'exprime \`a l'aide d'\'el\'ements $\bv\bt^2\bv\inv$ pour des
couples $(\bv,\bt)$ tels que les \'el\'ements
$\bv\bt\tilde\bv$ soient dans $\Bred$ et aient des
images toutes distinctes dans $W$.
\end{proof}
\begin{corollaire} (de \ref{noyau de N}) \label{N(b)=N(b')}
Si $\bb$ et $\bb'$ sont des \'el\'ements de $B_W$, on a $N(\bb)=N(\bb')$ si et
seulement si $\bb\inv\bb'\in D(P_W)$.
\end{corollaire}
\begin{proof} Si deux \'el\'ements ont m\^eme image par $N$, ils ont m\^eme
image dans $W$. Donc le quotient $\bb\inv\bb'$ est dans $P_W$ et
est dans le noyau de la restriction de $N$ \`a $P_W$ qui est un homomorphisme.
D'o\`u l'implication directe. R\'eciproquement,
si $\bb\inv \bb'\in D(P_W)$ alors en particulier $\bb$ et $\bb'$ ont m\^eme
image $w$
dans $W$. Si on pose $N(\bb)=(w,A)$ et $N(\bb')=(w,A')$ on a $1=N(\bb\inv
\bb')=(1,A'-A)$, donc $A=A'$.
\end{proof}
\begin{remarque}
\begin{enumerate}
\item La restriction via l'homomorphisme canonique de $N$ au
mono\"\i de de tresses $\BW$ est croissante (\`a valeurs dans $\bbN T$)
si l'on prend comme relation d'ordre 
dans le mono\"\i de la relation de divisibilit\'e \`a gauche
et dans $\bbZ T$ la relation donn\'ee par le fait que les \'el\'ements
de $\bbN T$ sont plus grands que 0.
\item
L'image du mono\"\i de $\BW$ par $N$ est un sous-mono\"\i de de $\bbN T$
qui engendre le groupe $\bbZ T$. Question: d\'eterminer ce
sous-mono\"\i de. M\^eme question pour l'image des \'el\'ements purs
positifs.
\end{enumerate}
\end{remarque}
\section{Pr\'esentations de certains sous-groupes du groupe de tresses}
Nous allons appliquer la m\'ethode de Reidemeister-Schreier pour donner
une pr\'esentation de certains sous-groupes de $B_W$ et en particulier du
groupe de tresses pur. Soit $I\subset S$. Notons $D_I$ l'image r\'eciproque de
$W_I$ dans $B_W$. Ce sous-groupe contient le groupe de tresses pur et le
sous-groupe parabolique $B_\bI$ de $B_W$ o\`u on a not\'e
$\bI$ la partie de $\bS$ dont
l'image dans $W$ est $I$. On dit qu'un \'el\'ement $\bw\in\Bred$ est
$\bI$-r\'eduit
si pour tout $\bs\in\bI$ on a $\bs\bw\in\Bred$. Ceci est \'equivaut \`a dire que
l'image $w$ de $\bw$ dans $W$ est $I$-r\'eduite, c'est-\`a-dire est de longueur
minimale dans sa classe \`a droite modulo $W_I$. On a de m\^eme une notion
d'\'el\'ements r\'eduits-$\bI$ et d'\'el\'ements $\bI$-r\'eduits-$\bJ$ pour deux parties
$\bI$ et $\bJ$ de $\bS$.
Commen\c cons par g\'en\'eraliser \ref{generateurs de P_W}.
\begin{proposition} \label{generateurs de D_I}
Pour tout choix d'un ensemble de couples $(\bw,\bs)\in\Bred\times\bS$ tel
que les \'el\'ements
$\bw\bs\tilde\bw\in\Bred$ parcourent l'ensemble des relev\'es
des r\'eflexions $I$-r\'eduites de $W$,
les \'el\'ements $\bw\bs^2\bw\inv$ et $I$ engendrent $D_I$.
\end{proposition}
\begin{proof}
Le groupe $D_I$ est engendr\'e par $\bI$ et $P_W$, donc par $\bI$ et un ensemble
de $\bw\bs^2\bw\inv$ comme dans \ref{generateurs de P_W}. Il suffit de voir que
l'on peut retirer de cet ensemble les \'el\'ements tels que $\bw\bs\tilde\bw$
ne soit pas $\bI$-r\'eduit. Par r\'ecurrence sur la longueur de $\bw$, il suffit
de voir que si $\bw\bs\tilde\bw$ n'est pas $\bI$-r\'eduit
$\bw\bs^2\bw\inv$ s'\'ecrit
comme produit d'\'el\'ements de $\bI$ et
d'\'el\'ements analogues $\bw'\bs^{\prime2}\bw^{\prime-1}$ avec
$l(\bw')<l(\bw)$ ou de leurs inverses.
Si $\bw\bs\tilde\bw$ n'est pas $\bI$-r\'eduit,
il a une \'ecriture qui commence par
$\bs'\in\bI$ et donc par \ref{Dyer} s'\'ecrit $\bs'\bv\br\tilde\bv\bs'$ avec
$\br\in\bS$; par les
lemmes \ref{un seul generateur par reflexion} et \ref{ws tilde w reduit}
$\bw\bs^2\bw\inv$ est conjugu\'e \`a $\bs'\bv\br^2\bv\inv\bs^{\prime-1}$
par des \'el\'ements de la forme
$\bu\bt^2\bu\inv$ avec $\bt\in\bS$,
$l(\bu)<l(\bw)$ et $\bu\bt\tilde\bu\in\Bred$.
Ceci donne le r\'esultat.
\end{proof}
Cherchons maintenant une pr\'esentation de $D_I$. 
Pour cela nous appliquons la m\'ethode de Reidemeister-Schreier.
Un ensemble de repr\'esentants du quotient $D_I\backslash B_W$
est form\'e par les \'el\'ements
$\bI$-r\'eduits de $\Bred$ car $D_I\backslash B_W\simeq W_I\backslash W$.
On note $[\bb]$ le repr\'esentant de la classe $\bb D_I$.
Les \'el\'ements $\bw\bs[\bw\bs]\inv$ o\`u $\bw$
parcourt l'ensemble des \'el\'ements
$\bI$-r\'eduits et o\`u $\bs$ parcourt $\bS$ 
forment donc un ensemble g\'en\'erateur de
$D_I$. Si $\bw\bs$ est dans $\Bred$ et
est $\bI$-r\'eduit le g\'en\'erateur obtenu vaut 1.
Si $\bw\bs$ est dans $\Bred$ mais non $\bI$-r\'eduit, cela signifie qu'il existe
$\bt\in\bI$ et $\bv\in\Bred$ tel que $\bw\bs=\bt\bv$. Alors $\bv$ est
$\bI$-r\'eduit et est \'egal \`a $[\bw\bs]$. Le g\'en\'erateur
obtenu est $\bt$ (ceci inclut le cas o\`u $\bw=1$, donc $\bt$ parcourt tout
$\bI$). Si $\bw\bs$ n'est pas dans $\Bred$ alors $\bw=\bv\bs$,
l'\'el\'ement
$\bv$ est $\bI$-r\'eduit-$\{\bs\}$ et est  \'egal \`a $[\bw\bs]$
car $\bw\bs=\bv\bs^2=(\bv\bs^2\bv\inv)\bv$. Le g\'en\'erateur
obtenu est $\bv\bs^2\bv\inv$. On obtient donc que $D_I$ est engendr\'e par $\bI$
et par les \'el\'ements $\bw\bs^2\bw\inv$ o\`u $\bs$ est dans $\bS$ et
$\bw\bs\in\Bred$ est $\bI$-r\'eduit.

\begin{definition}\label{a_v,s}
Pour $\bw\in\Bred$ et $\bs\in\bS$ tels que
$\bw\bs\in\Bred$ soit $\bI$-r\'eduit, on pose $a_{\bw,\bs}=\bw\bs^2\bw\inv$.
\end{definition}

La m\'ethode de Reidemeister-Schreier
dit que pour chaque repr\'esentant de $B_W/D_I$ dans l'ensemble
choisi
et pour chaque relation de tresses entre deux \'el\'ements de $\bS$ on obtient une
relation entre les g\'en\'erateurs de $D_I$ par le proc\'ed\'e de r\'e\'ecriture
et que les relations obtenues forment un
ensemble complet de relations. 

Autrement dit pour chaque repr\'esentant
$\bb$, c'est-\`a-dire chaque
\'el\'ement $\bI$-r\'eduit de $\Bred$, et chaque paire
$(\bs,\bt)$, telle que $m_{s,t}$ soit fini, d'\'el\'ements de
$\bS$, nous devons r\'e\'ecrire la relation
$$\bb\underbrace{\bs\bt\ldots}_{m_{s,t}}=
\bb\underbrace{\bt\bs\ldots}_{m_{s,t}}.\eqno (*)$$

Rappelons le proc\'ed\'e de r\'e\'ecriture:
\`a chaque \'el\'ement $\bI$-r\'eduit 
$\bb\in\Bred$ et chaque \'el\'ement $\bs\in\bS$ est associ\'e un g\'en\'erateur
$\bb\bs[\bb\bs]\inv$ de $D_I$.
Le proc\'ed\'e de r\'e\'ecriture
consiste \`a remplacer le produit $\bb\bs$ par
$(\bb\bs[\bb\bs]\inv)[\bb\bs]$,
ce qui permet de proche en
proche d'\'ecrire tout \'el\'ement de $B_W$ comme produit
de g\'en\'erateurs de $D_I$ et d'un repr\'esentant de $B_W/D_I$.
 
En appliquant cette r\'e\'ecriture aux deux membres de $(*)$ on obtient une
relation entre les g\'en\'erateurs de $D_I$ et quand $\bb$, $\bs$ et $\bt$
varient on obtient un ensemble complet de relations.

Pour faire le calcul nous utilisons la notation suivante:
\begin{notation} \label{wst} Un produit de
$i\leq m_{s,t}$
facteurs alternativement $\bs$ et $\bt$ sera not\'e
$\wst i$.
\end{notation}

Pour all\'eger les notations nous
\'ecrirons $m$ pour $m_{s,t}$.
Quitte \`a \'echanger $\bs$ et $\bt$,
on peut \'ecrire $\bb=\bb_0.\wst i$ avec
$0\leq i\leq m$ et $i$ maximal, c'est-\`a-dire que $\bb_0$ est
r\'eduit-$\{\bs,\bt\}$.
Pour $0\leq j<m$,
nous \'ecrirons $a_{\bs,\bt}^{(j)}$
\`a la place de $a_{\bb_0\wst j,\br}$ o\`u
$\br$ vaut $\bs$ si $j$ est pair et $\bt$ sinon
(c'est-\`a-dire, sauf dans le cas $j=m-1$, que $\br$ est celui des deux
g\'en\'erateurs $\bs$ ou $\bt$ qui divise \`a droite $\wst {j+1}$
dans le mono\"\i de $\BW$).

Nous allons subdiviser le calcul en plusieurs cas suivant la valeur de $i$ et
les propri\'et\'es de $\bb_0$.

Cas 1): Cas o\`u $i=0$.

Cas 2):
On a $i\geq 1$ et
$\bb_0\bt$ est $\bI$-r\'eduit
(par hypoth\`ese $\bb_0\bs$ est $\bI$-r\'eduit, puisque $i\geq 1$).

Cas 3):
On a $i\geq 1$ et
$\bb_0\bt=\bs'\bb_0$ avec $\bs'\in\bI$ (et $\bb_0\bs$ est $\bI$-r\'eduit).

Pour appliquer le proc\'ed\'e  de r\'e\'ecriture
dans les trois cas ci-dessus, nous aurons
besoin des deux lemmes suivants dont le premier est bien connu.
Dans ces deux lemmes on consid\`ere deux \'el\'ements $\bs$ et $\bt$ de $\bS$ tels
que $m_{s,t}$ est fini.
\begin{lemme}\label{b s'ajoute a s et t}
Si $\bb\in\Bred$ est tel que $\bb\bs$ et $\bb\bt$ sont dans $\Bred$
alors $\bb\bw_{\{\bs,\bt\}}\in\Bred$.
\end{lemme}
\begin{lemme}\label{b w_s,t I-reduit}
Si $\bb\in\Bred$ est un \'el\'ement
r\'eduit-$\{\bs,\bt\}$ et si $\bb\bs$ et $\bb\bt$ sont $\bI$-r\'eduits
alors $\bb\bw_{\{\bs,\bt\}}$ est $\bI$-r\'eduit.
\end{lemme}
\begin{proof}
Si $\bb\bw_{\{\bs,\bt\}}$ n'est pas $\bI$-r\'eduit, notons $i$ le plus petit
indice tel que $\bb.\wst i$ ne soit pas $\bI$-r\'eduit. On a $i>1$ et
il existe un \'el\'ement
$\bs'\in\bI$ tel que $\bb.\wst i=\bs'\bb.\wst {i-1}$. Cette \'egalit\'e
implique que $\bb.\wst i$ est divisible \`a droite \`a la fois par $\bs$ et $\bt$
(car $i-1>0$), donc est divisible par $\bw_{\{\bs,\bt\}}$, ce qui implique
$i=m_{s,t}$ car $\bb$ est r\'eduit-$\{\bs,\bt\}$. On peut alors simplifier par 
$\wst {i-1}$ et l'on obtient que $\bs'\bb$ est \'egal \`a $\bb\bs$ ou
$\bb\bt$ (selon la parit\'e de $i$), ce qui est contraire \`a l'hypoth\`ese.
\end{proof}

Cas 1):
Si $i=0$, si $\bb_0\bs$ et $\bb_0\bt$ sont tous deux $\bI$-r\'eduits, la relation
obtenue est triviale $1=1$.
Si $\bb_0\bs=\bs'\bb_0$ avec $\bs'\in\bI$ et si $\bb_0\bt$
est $\bI$-r\'eduit on obtient la relation triviale $\bs'=\bs'$.
Si $\bb_0\bs=\bs'\bb_0$ et $\bb_0\bt=\bt'\bb_0$ avec $\bs'$ et $\bt'$ dans
$\bI$, on obtient la relation de tresses entre $\bs'$ et $\bt'$ (remarquer que
$m_{\bs,\bt}=m_{\bs',\bt'}$ dans ce cas). Ce dernier cas nous donne toutes les
relations de tresses entre \'el\'ements de $\bI$ car il contient en
particulier le cas o\`u $\bb_0=1$.

Cas 2):
Si $\bb_0\bt$ est $\bI$-r\'eduit (et $i\geq 1$).
Dans un des deux membres de $(*)$, le produit de $\bb_0\wst i$ par le
premier facteur $\bs$ ou $\bt$ est dans $\Bred$.
Le proc\'ed\'e de r\'e\'ecriture est alors trivial jusqu'\`a 
arriver au produit $\bb_0\wmst$. La r\'e\'ecriture du produit suivant qui
consiste \`a multiplier $\bb_0\wmst$ par $\bs$ ou $\bt$ suivant la parit\'e donne
$a_{\bt,\bs}^{(m-1)}\bb_0\wts {m-1}$.
On continue de proche en proche
pour obtenir le premier membre sous la forme
$$a_{\bt,\bs}^{(m-1)}a_{\bt,\bs}^{(m-2)}\ldots
a_{\bt,\bs}^{(m-i)}\bb_0\wts{m-i}.$$
Dans l'autre membre le m\^eme type de calcul donne
$$a_{\bs,\bt}^{(i-1)}a_{\bs,\bt}^{(i-2)}\ldots
a_{\bs,\bt}^{(0)}\bb_0\wts {m-i}.$$
La relation obtenue est donc
$$a_{\bt,\bs}^{(m-1)}
a_{\bt,\bs}^{(m-2)}\ldots a_{\bt,\bs}^{(m-i)}=
a_{\bs,\bt}^{(i-1)}a_{\bs,\bt}^{(i-2)}\ldots
a_{\bs,\bt}^{(0)}.$$

Cas 3):
Dans un des deux membres,
comme dans le cas pr\'ec\'edent, le proc\'ed\'e de r\'e\'ecriture 
est trivial jusqu'\`a 
arriver au produit $\bb_0\wmst=\bs'\bb_0\wst {m-1}$.
La r\'e\'ecriture des produits suivants se passe comme dans le cas pr\'ec\'edent
et on obtient
$$\bs'a_{\bs,\bt}^{(m-2)}a_{\bs,\bt}^{(m-3)}\ldots
a_{\bs,\bt}^{(m-i-1)}\bb_0\wst{m-i-1}.$$
Dans l'autre membre la r\'e\'ecriture des premiers
produits se passe comme dans le premier cas,
jusqu'\`a arriver \`a r\'e\'ecrire le
produit $\bb_0\bt$ qui donne $\bs'\bb_0$. Les r\'e\'ecritures suivantes sont
triviales et on trouve
$$a_{\bs,\bt}^{(i-1)}a_{\bs,\bt}^{(i-2)}\ldots
a_{\bs,\bt}^{(0)}\bs'\bb_0\wst {m-i-1}.$$
La relation obtenue est donc
$$\bs'a_{\bs,\bt}^{(m-2)}a_{\bs,\bt}^{(m-3)}\ldots
a_{\bs,\bt}^{(m-i-1)}=
a_{\bs,\bt}^{(i-1)}a_{\bs,\bt}^{(i-2)}\ldots
a_{\bs,\bt}^{(0)}\bs'.$$

Pour pouvoir r\'esumer les consid\'erations pr\'ec\'edentes dans un \'enonc\'e nous
aurons besoin de faire varier $\bb_0$ dans ce qui pr\'ec\`ede: nous noterons
$a_{\bb_0,\bs,\bt}^{(i)}$ au lieu de $a_{\bs,\bt}^{(i)}$, c'est-\`a-dire 
$a_{\bb_0,\bs,\bt}^{(i)}= a_{\bb_0\underbrace{\scriptstyle\bs\bt\ldots}_i,\br}$,
avec la notation \ref{a_v,s}, o\`u $\br$ vaut $\bs$ ou $\bt$ suivant que $i$ est
pair ou impair.
\begin{proposition} \label{presentation de D_I}
Le groupe $D_I$ admet une pr\'esentation o\`u les g\'en\'erateurs sont les
\'el\'ements de $\bI$ et les
$a_{\bb,\bs}$, o\`u $\bs\in\bS$ et $\bb\bs\in\Bred$ est $\bI$-r\'eduit,
et o\`u les relations sont 
les relations de tresses entre \'el\'ements de $\bI$ et les relations
\begin{align*}
a_{\bb_0,\bs,\bt}^{(m-1)}a_{\bb_0,\bs,\bt}^{(m-2)}
\ldots a_{\bb_0,\bs,\bt}^{(m-i)}&=
a_{\bb_0,\bt,\bs}^{(i-1)}a_{\bb_0,\bt,\bs}^{(i-2)}\ldots
a_{\bb_0,\bt,\bs}^{(0)}\text{ pour }i=1,\ldots,m& (1)\\
\bs'a_{\bb_0,\bs,\bt}^{(m-2)}a_{\bb_0,\bs,\bt}^{(m-3)}\ldots
a_{\bb_0,\bs,\bt}^{(m-i-1)}&=
a_{\bb_0,\bs,\bt}^{(i-1)}a_{\bb_0,\bs,\bt}^{(i-2)}\ldots
a_{\bb_0,\bs,\bt}^{(0)}\bs'
\text{ pour } i=1,\ldots, m-1,& (2)
\end{align*}
o\`u $(\bs,\bt)$ est un couple arbitraire d'\'el\'ements de $\bS$ tel que
$m=m_{s,t}$ est fini,
o\`u dans $(1)$ $\bb_0\in\Bred$ d\'ecrit l'ensemble des \'el\'ements
r\'eduits-$\{\bs,\bt\}$ tels que $\bb_0\bs$ et $\bb_0\bt$ soient $\bI$-r\'eduits
et dans $(2)$ $\bb_0\in\Bred$ d\'ecrit l'ensemble des \'el\'ements
r\'eduits-$\{\bs,\bt\}$ tels que $\bb_0\bs$
soit $\bI$-r\'eduit et que $\bb_0\bt=\bs'\bb_0$ avec $\bs'\in\bI$.
\end{proposition}

Donnons le cas particulier du groupe de tresses pur:
\begin{corollaire}\label{presentation du groupe pur}
Le groupe $P_W$ a une pr\'esentation o\`u les g\'en\'erateurs sont les
$a_{\bb,\bs}$ o\`u $\bs\in\bS$ et $\bb\bs\in\Bred$ et o\`u les relations sont,
avec les notations de la proposition,
$$a_{\bb_0,\bs,\bt}^{(m-1)}
a_{\bb_0,\bs,\bt}^{(m-2)}\ldots a_{\bb_0,\bs,\bt}^{(m-i)}=
a_{\bb_0,\bt,\bs}^{(i-1)}a_{\bb_0,\bt,\bs}^{(i-2)}\ldots
a_{\bb_0,\bt,\bs}^{(0)},$$
pour chaque couple $(\bs,\bt)$ d'\'el\'ements de $\bS$ tel que $m_{s,t}$ est fini,
pour chaque $\bb_0\in\Bred$
r\'eduit-$\{\bs,\bt\}$ et chaque $i$ variant de 1 \`a $m-1$, o\`u l'on a pos\'e
$m=m_{s,t}$.
\end{corollaire}
\begin{proof} Il suffit d'appliquer ce qui pr\'ec\`ede dans le cas o\`u $I$
est vide.
\end{proof}
\begin{corollaire} \label{Produit semi-direct}
\begin{enumerate}
\item Le groupe $B_\bI$ est isomorphe au groupe de tresses $B_{W_I}$.
\item
Si $U_\bI$ est le sous-groupe normal de $D_I$ engendr\'e par les
\'el\'ements $a_{\bb,\bs}$ comme dans \ref{presentation de D_I} alors
$D_I=U_\bI\rtimes B_\bI$.
\end{enumerate}
\end{corollaire}
\begin{proof}
On peut d\'efinir un homomorphisme $h:D_I\to B_{W_I}$
qui envoie les \'el\'ements de $\bI$ sur les g\'en\'erateurs correspondant de
$B_{W_I}$ et $a_{\bb,\bs}$ sur 1 car c'est compatible avec les relations
dans $D_I$. L'homomorphisme $h$ est \'evidemment surjectif et son noyau
contient $U_\bI$. D'autre part il existe aussi un homomorphisme
$j:B_{W_I}\to B_\bI$ qui envoie les g\'en\'erateurs de $B_{W_I}$ sur les
\'el\'ements de $\bI$. On a $h\circ j=\Id$ et comme $j$ est clairement
surjectif on obtient (i).
Ceci montre aussi que $B_\bI$ ne rencontre pas le noyau de $h$. Comme de
plus $B_\bI$ et $U_\bI$ engendrent $D_I$ et que $U_\bI$ est normal et
inclus dans le noyau de $h$ on a (ii).
\end{proof}
\begin{remarque} Si $W$ est fini l'ensemble des r\'eflexions images
des \'el\'ements $\bw\bs\tilde\bw$ o\`u $\bw\bs$ est $\bI$-r\'eduit
est exactement l'ensemble des r\'eflexions qui ne
sont pas dans $W_I$ c'est-\`a-dire $\overline N(w_Iw_S)$. En effet
si $ws$ est $I$-r\'eduit, il existe une \'ecriture
r\'eduite de $w_Iw_S$ qui commence par $ws$, donc la r\'eflexion $wsw\inv$
est bien dans $\overline N(w_Iw_S)$. R\'eciproquement toute \'el\'ement de
$\overline N(w_Iw_S)$ s'\'ecrit $wsw\inv$ o\`u $w$ est r\'eduit-$s$ et o\`u $ws$ est le
d\'ebut d'une \'ecriture r\'eduite de $w_Iw_S$; en particulier $ws$ est $I$-r\'eduit.
Ce r\'esultat n'est pas vrai pour un groupe de Coxeter quelconque comme le montre
l'exemple suivant dans $\tilde A_2$, o\`u on note $S=\{r,s,t\}$. On prend
$I=\{r,s\}$. La r\'eflexion $srtrs$ n'est pas dans $W_I$, mais elle ne s'\'ecrit
pas $ws'w\inv$ avec pour $ws'$ un \'el\'ement $I$-r\'eduit.
En effet dans ce cas on
aurait $rsws'=trsw$ avec des longueurs qui s'ajoutent, ce qui implique que cet
\'el\'ement est divisible par $rs$ et par $t$. Ceci est impossible car $\overline
N(w)$ contiendrait alors $rsr$ et $t$ et donc
aussi toutes les r\'eflexions du groupe
di\'edral infini engendr\'e par ces deux \'el\'ements.
\end{remarque}
\begin{corollaire} \label{devissage} Consid\'erons une suite de parties de
$\bS$ embo{\^\i}t\'ees $\emptyset=\bI_0\subsetneq
\bI_1\subsetneq \bI_2\ldots\subsetneq \bI_n=\bS$ et notons $U_j$ le groupe
$U_\bI$ de \ref{Produit semi-direct}  quand le groupe ambient est $B_{\bI_j}$
et que $\bI=\bI_{j-1}$; on a 
$P_W=U_n\rtimes(U_{n-1}\rtimes(\cdots\rtimes(U_2\rtimes
U_1)\ldots))$.
\end{corollaire}
\begin{proof} Comme $U_n\subset P_W$, le corollaire
\ref{Produit semi-direct} appliqu\'e avec $\bI=\bI_{n-1}$
implique que $P_W=U_n\rtimes P_{W_{\bI_{n-1}}}$.
On obtient le r\'esultat par r\'ecurrence \`a partir de cette \'egalit\'e.
\end{proof}
\section{Application aux groupes de type fini}
Nous allons pr\'eciser la pr\'esentation pr\'ec\'edente quand $W$ est fini en
utilisant \ref{devissage} pour obtenir
une pr\'esentation meilleure que dans
\ref{presentation du groupe pur} du groupe de tresses pur.
Nous \'etudions d'abord la question de savoir si $U_\bI$ est engendr\'e par les
$a_{\bb,\bs}$ (sans qu'il soit besoin de prendre la cl\^oture normale).
Il en est ainsi si dans les relations de type $(2)$ de
\ref{presentation de D_I} on peut obtenir toutes les conjugaisons des \'el\'ements
$a_{\bb,\bs}$ par les \'el\'ements $\bs'\in\bI$.
Le lemme suivant, valable sans hypoth\`ese de finitude de $W$,
pr\'ecise quelles conjugaisons on peut obtenir.
\begin{lemme}\label{s'*a_b1s*s' inv} Soit $\bs'\in\bI$ et soient
$\bb\in\Bred$
et $\bs\in\bS$ tels que $\bb\bs$ soit r\'eduit et $\bI$-r\'eduit;
on a une relation du type $(2)$ dans \ref{presentation de D_I}
faisant intervenir $\bs'$ et $a_{\bb,\bs}$ si et seulement si
$\bb\inv\bs'\bb\in B_{\{\bs,\bt\}}$ pour un certain $\bt$ tel que $m_{\bs,\bt}$
soit fini.
\end{lemme}
\begin{proof}
Si on a une relation de type $(2)$, on a vu que $\bb$ s'\'ecrit
$\bb_0\wst i$ o\`u $\bb_0\bs$ est $\bI$-r\'eduit, $\bb_0\bt=\bs'\bb_0$ et
$m_{\bs,\bt}$ est fini.
Alors $\bb\inv\bs'\bb=(\wst i)\inv \bt\wst i\in B_{\{\bs,\bt\}}$.
R\'eciproquement si $\bb\inv\bs'\bb\in B_{\{\bs,\bt\}}$, \'ecrivons $\bb=\bb_0\wst i$ 
avec $\bb_0$ r\'eduit-$\{\bs,\bt\}$. Alors
$\bb_0\inv\bs'\bb_0\in B_{\{\bs,\bt\}}$, \ie, $\bs'\bb_0=\bb_0\bx$ avec
$\bx\in B_{\{\bs,\bt\}}$. Or $\bs'\bb_0$ est dans $\Bred$ car
$\bb$ est $\bI$-r\'eduit. On en d\'eduit, par exemple en prenant les images dans
$W$, que $\bx$ est dans $B_{\{\bs,\bt\}}^+$ et est de
longueur 1. Comme $\bx$ n'est pas \'egal \`a $\bs$
car $\bb\bs$ est $I$-r\'eduit, donc aussi
$\bb_0\bs$, on a $\bx=\bt$ c'est-\`a-dire $\bs'\bb_0=\bb_0\bt$.
La conjugaison de l'\'el\'ement $a_{\bb,\bs}$ 
par $\bs'$ appara{\^\i}t donc bien dans les relations $(2)$ de
\ref{presentation de D_I}.
\end{proof}
Remarquons que dans l'\'enonc\'e pr\'ec\'edent $\bb\inv\bs'\bb$ ne peut valoir ni 1 ni
$\bs$, donc qu'il n'y a pas d'ambigu{\"\i}t\'e sur $\bt$.

Nous allons utiliser le lemme pr\'ec\'edent pour obtenir 
un r\'esultat sur $U_\bI$ quand $W$
est de type $A_n$, $B_n$ ou $I_2(m)$.
Rappelons que quand $W$ est fini, pour tout $\bI$ l'\'el\'ement
$\bb^\bI=\bw_\bI\inv\bw_\bS\in\Bred$
est le plus grand \'el\'ement $\bI$-r\'eduit de $\Bred$
(au sens que tout \'el\'ement $\bI$-r\'eduit en est un diviseur \`a gauche).
\begin{proposition} \label{U_I normal}
Supposons $W$ fini et soit $\bI$ tel que le plus grand \'el\'ement $\bI$-r\'eduit
$\bb^\bI$ de $\Bred$ ait une seule \'ecriture comme produit de g\'en\'erateurs
dans $\BW$. Alors
\begin{enumerate}
\item Il n'y a pas de relation de type \ref{presentation de D_I},
$(1)$ dans $D_I$.
\item
le groupe $U_\bI$ est engendr\'e par
les $a_{\bb,\bs}$.
\end{enumerate}
\end{proposition}
\begin{proof}
Comme tout \'el\'ement $\bI$-r\'eduit divise $\bb^\bI$ et que celui-ci n'a qu'une
\'ecriture dans $\BW$, on ne peut pas avoir $\bb\bs$ et $\bb\bt$ tous deux
$\bI$-r\'eduits, donc le cas $(1)$ de \ref{presentation de D_I}
ne se produit pas, d'o\`u (i).

On note $\bs_1$ le seul \'el\'ement de $\bS$ qui divise $\bb^\bI$ dans $\BW$.
Soit $\bs'\in\bI$; on a $\bs'\bb^\bI=\bb^\bI\br$ pour
un certain $\br\in\bS$ car $\bw_\bI$ (resp. $\bw_\bS$)
conjugue $\bI$ (resp. $\bS$) sur lui-m\^eme.
Donc $\bs'\bb^\bI$ est divisible \`a gauche par $\bs'$ et $\bs_1$, donc par
$\bw_{\{\bs',\bs_1\}}$. Donc
$\bb^\bI=\bw_{\bs_1,\bs'}^{(m_{\bs_1,\bs'}-1)}\bb_1$ et,
apr\`es simplification par $\bw_{\bs_1,\bs'}^{(m_{\bs_1,\bs'}-1)}$,
l'\'egalit\'e $\bs'\bb^\bI=\bb^\bI\br$ devient
$\bs'_2\bb_1=\bb_1\br\in\Bred$, o\`u $\bs'_2$ vaut $\bs'$ ou $\bs_1$.
Le m\^eme argument it\'er\'e montre finalement que $\bb^\bI$ est de la forme
$\bb^\bI=\prod_{i=1}^{i=k}\bw_{\bs_i,\bs'_i}^{(m_{\bs_i,\bs'_i}-1)}$ avec
$\bs'_1=\bs'$ et o\`u $\bs'_{i+1}$ vaut $\bs'_i$ ou $\bs_i$, pr\'ecis\'ement
$\bs'_i\bw_{\{\bs_i,\bs'_i\}}=\bw_{\{\bs_i,\bs'_i\}}\bs'_{i+1}$.
L'\'ecriture
$\bb^\bI=\prod_{i=1}^{i=k}\bw_{\bs_i,\bs'_i}^{(m_{\bs_i,\bs'_i}-1)}$
\'etant une \'ecriture dans $\Bred$ fournit
l'\'ecriture unique de $\bb^\bI$.

Soient alors
$\bb\in\Bred$ et $\bs\in\bS$ tels que $\bb\bs$ soit $\bI$-r\'eduit.
L'\'el\'ement $\bb\bs$ divise $\bb^\bI$ donc est \'egal \`a
$\left[\prod_{i=1}^{i=h-1}\bw_{\bs_i,\bs'_i}^{(m_{\bs_i,\bs'_i}-1)}\right]
\bw_{\bs_h,\bs'_h}^{(j)}$, pour un certain $h$ et un certain
$j<m_{s_h,s_{h'}}$, avec $\bs'=\bs'_1$ et $\bs$
\'egal au dernier terme de $\bw_{\bs_h,\bs'_h}^{(j)}$.
Ceci montre que $\bs'\bb=
\left[\prod_{i=1}^{i=h-1}\bw_{\bs_i,\bs'_i}^{(m_{\bs_i,\bs'_i}-1)}\right]
\bs'_h\bw_{\bs_h,\bs'_h}^{(j-1)}$, donc
$\bb\inv\bs'\bb=
(\bw_{\bs_h,\bs'_h}^{(j-1)})\inv\bs'_h\bw_{\bs_h,\bs'_h}^{(j-1)}$
est dans un groupe de la forme
$B_{\{\bs,\bt\}}$ avec $m_{\bs,\bt}$ fini
et on peut appliquer le lemme \ref{s'*a_b1s*s' inv}; donc par les relations de
type (2) dans \ref{presentation de D_I} le conjugu\'e de $a_{\bb,\bs}$ par
$\bs'$ est un produit d'\'el\'ements de la m\^eme forme. Le groupe engendr\'e
par les $a_{\bb,\bs}$ est donc normal, donc \'egal \`a $U_\bI$.
\end{proof}

Les cas o\`u la
proposition pr\'ec\'edente s'applique sont 
exactement les cas $A_n$ et $B_n$
avec $\bI=\{\bs_1,\bs_2\ldots,\bs_{n-1}\}$
(la double liaison \'etant entre $s_1$
et $s_2$ pour $B_n$) et $I_2(m)$ o\`u $\bI$ est l'un des
deux g\'en\'erateurs.

Nous donnons maintenant cas par cas,
quand $W$ est de type $A_n$, $B_n$, $D_n$ ou
$I_2(m)$, des pr\'esentations de $D_I$ permettant
par r\'ecurrence d'obtenir une pr\'esentation du groupe de tresses pur.

Commen\c cons par rappeler le cas $A_n$ qui est bien connu.
\begin{subsection}{Type $A_n$} Le diagramme de Coxeter est
$$\dbulnode{s_1}\bar\cdots\bar\dbulnode{s_{n-1}}\bar\dbulnode{s_n}.$$
On prend $\bI=\{\bs_1,\ldots,\bs_{n-1}\}$. Alors
$\bb^\bI=\bs_n\bs_{n-1}\ldots\bs_1$. Les $a_{\bb,\bs}$ sont donc les
$a_i=(\lexp{\bs_n\ldots\bs_{i+1}}\bs_i)^2$ pour $i=1,\ldots,n$.
Les relations dans $D_I$ sont donc les relations de tresses dans $\bI$
et
$$\begin{cases}\bs_ja_i\bs_j\inv=a_i&\text{ si } i\neq j,j+1\\
         \bs_ia_i\bs_i\inv=a_{i+1}&\\
         \bs_ia_ia_{i+1}\bs_i\inv=a_ia_{i+1}&\\
\end{cases}$$
\end{subsection}

\begin{subsection}{Type $B_n$}
Le diagramme de Coxeter est
$$\dbulnode{s_1}\dbar\dbulnode{s_2}\bar\cdots\bar\dbulnode{s_n}.$$
On prend $\bI=\{\bs_1,\ldots,\bs_{n-1}\}$. Alors
$\bb^\bI=\bs_n\bs_{n-1}\ldots\bs_1\bs_2\ldots\bs_n$.
Les $a_{\bb,\bs}$ sont donc les
$a_i=(\lexp{\bs_n\ldots\bs_{i+1}}\bs_i)^2$ pour $i=1,\ldots,n$ et les
$b_i=(\lexp{\bs_n\ldots\bs_1\bs_2\ldots\bs_{i-1}}\bs_i)^2$ pour $i=2,\ldots,n$.
Les relations dans $D_I$ sont donc les relations de tresses dans $\bI$
et
$$\begin{cases}\bs_ja_i\bs_j\inv=a_i&\text{ si }i\neq j,j+1\\
         \bs_ia_i\bs_i\inv=a_{i+1}& \text{ pour }i\neq 1\\
         \bs_ia_ia_{i+1}\bs_i\inv=a_ia_{i+1}&\text{ pour }i\neq 1\\
         \bs_jb_i\bs_j\inv=b_i&\text{ si }i\neq j,j+1\text{ et }i\geq2\\
         \bs_ib_{i+1}\bs_i\inv=b_i&\text{ pour }i\neq 1\\
         \bs_ib_{i+1}b_i\bs_i\inv=b_{i+1}b_i&\text{ pour }i\neq 1\\
         \bs_1b_2\bs_1\inv=a_2&\\
         \bs_1b_2a_1\bs_1\inv=a_1a_2&\\
         \bs_1b_2a_1a_2\bs_1\inv=b_2a_1a_2&\\
\end{cases}$$
\end{subsection}

\begin{subsection}{Type $I_2(m)$} On note $\bs$ et $\bt$
les deux g\'en\'erateurs et on prend $\bI=\{\bs\}$. On a $\bb^\bI=
\wst{m-1}$ et les $a_{\bb,\bs}$ sont les $a_i=(\lexp{\wst{i-1}}\br)^2$
o\`u $\br$ vaut $\bs$ ou $\bt$ selon la parit\'e de $i$ et $i=1,\ldots,m-1$.
La proposition \ref{presentation de D_I} donne des relations de type $(2)$
avec $\bb_0=1$ et $\bs'=\bs$. On obtient donc comme uniques relations
$$\bs a_{m-1}a_{m-2}\ldots a_{m-i}\bs\inv=a_i\ldots a_2a_1, \text{ pour
}i=1,\ldots,m-1.$$
\end{subsection}

\begin{subsection}{Type $D_n$}
Le diagramme de Coxeter est
$$\dbulnode{s_2}\bar
\vertbar{\dbulnode{s_3}}{\bulnode{\scriptstyle s_{2'}}}\bar
\dbulnode{s_4}\bar\cdots\bar\dbulnode{s_n}.$$
On suppose $n\geq 2$.
On note les g\'en\'erateurs $\bs_2,\bs_{2'},\bs_3,\ldots,\bs_n$ et
on prend
$\bI=\bS-\{\bs_n\}$ si $n\geq 3$ et $\bI=\emptyset$ si $n=2$.
On a $\bb^\bI=\bs_n\bs_{n-1}\ldots\bs_3\bs_2\bs_{2'}
\bs_3\ldots\bs_n$. Si $n\geq 3$
cet \'el\'ement a exactement deux \'ecritures dans $\BW$,
on passe de l'une \`a l'autre en \'echangeant $\bs_2$ et $\bs_{2'}$.
Les \'el\'ements $a_{\bb,\bs}$ sont donc les
$a_i=(\lexp{\bs_n\ldots\bs_{i+1}}\bs_i)^2$ pour $i=2,2',3,\ldots,n$, avec la
convention que $i+1$ vaut $3$ si $i=2'$,
les
$b_i=(\lexp{\bs_n\ldots\bs_3\bs_2\bs_{2'}\bs_3\ldots\bs_{i-1}}\bs_i)^2$ pour
$i=3,\ldots,n$, ainsi que
les deux \'el\'ements $a_{\bs_n\ldots\bs_3\bs_{2'},\bs_2}$ et
$a_{\bs_n\ldots\bs_3\bs_2,\bs_{2'}}$, qui sont clairement \'egaux respectivement \`a
$a_2$ et $a_{2'}$.

Appliquons la proposition \ref{presentation de D_I}. Il y a des relations de
type $(1)$, les seuls triplets $(\bb_0,\bs,\bt)$ possibles \'etant
$(\bs_n\ldots\bs_4\bs_3,\bs_2,\bs_{2'})$ et
$(\bs_n\ldots\bs_4\bs_3,\bs_{2'},\bs_2)$, ce qui donne les relations d'\'egalit\'e
remarqu\'ees ci-dessus et la relation $a_2a_{2'}=a_{2'}a_2$.

Les relations de type (2) obtenues sont les suivantes,
o\`u $i$ et $j$ parcourent $\{2,2',3,\ldots,n\}$ et o\`u l'on fait encore la
convention que $i+1$ vaut $3$ si $i=2'$:
$$\begin{cases}\bs_ja_i\bs_j\inv=a_i&\text{ si }i\neq j,j+1\text{ et }
                                      \{i,j\}\neq\{2,2'\}\\
         \bs_ia_i\bs_i\inv=a_{i+1}&\\
         \bs_ia_ia_{i+1}\bs_i\inv=a_ia_{i+1}&\\
         \bs_jb_i\bs_j\inv=b_i&\text{ si }i\neq j,j+1 \text{ et }i\geq3\\
         \bs_ib_{i+1}\bs_i\inv=b_i&\text{ pour }i\geq 3\\
         \bs_ib_{i+1}b_i\bs_i\inv=b_{i+1}b_i&\text{ pour }i\geq 3\\
         \bs_2b_3\bs_2\inv=a_{2'}&\\
         \bs_2b_3a_{2'}\bs_2\inv=b_3a_{2'}&\\
\end{cases} $$
ainsi que les relations obtenues  \`a partir des deux derni\`eres lignes
en \'echangeant $a_2$ et $a_{2'}$ et en m\^eme temps
$\bs_2$ et $\bs_{2'}$.

Le groupe $D_I$ a donc une pr\'esentation o\`u les g\'en\'erateurs sont $\bI$,
les $a_i$ et les $b_i$ avec comme relations les relations de tresse dans $\bI$
la relation $a_2a_{2'}=a_{2'}a_2$ et les relations ci-dessus.
Les hypoth\`eses de la
proposition \ref{U_I normal} ne sont pas vraies pour
le type $D_n$. N\'eanmoins la propri\'et\'e (ii) de la proposition est vraie, d'apr\`es
les relations ci-dessus: le groupe $U_\bI$ est engendr\'e par les $a_{\bb,\bs}$.
\end{subsection}

Donnons un renseignement de plus sur $U_\bI$:
\begin{proposition} \label{U_I libre}
Dans les cas $A_n$, $B_n$ et $I_2(m)$, avec $\bI$ comme ci-dessus,
le groupe $U_\bI$
est le groupe libre engendr\'e par les $a_{\bb,\bs}$ o\`u $\bb\bs$ est
$\bI$-r\'eduit.
\end{proposition}
\begin{proof}
Dans le cas $A_n$ c'est un r\'esultat connu et dans le cas $I_2(m)$ c'est
imm\'ediat \`a partir de la pr\'esentation donn\'ee ci-dessus. \'Etudions le cas $B_n$.
La m\'ethode que nous employons peut d'ailleurs \^etre appliqu\'ee aussi \`a $A_n$.
On commence par changer d'ensemble de g\'en\'erateurs en posant
$x_i=b_nb_{n-1}\ldots b_2a_1a_2\ldots a_i$ pour $i=1,\ldots ,n$
et $y_i=b_nb_{n-1}\ldots b_{i+1}$ pour $i=1,\ldots,n-1$.
Le groupe $U_\bI$ est engendr\'e par les $x_i$ et les $y_i$ et $D_I$
a une pr\'esentation avec comme g\'en\'erateurs les $x_i$
(avec $i=1,\ldots,n$), les $y_i$ (avec $i=1,\ldots,n-1$) et les $\bs_i$ (avec
$i=1,\ldots,n-1$) et comme relations les relations de tresse entre les $\bs_i$
et les relations de type (1) qui s'\'ecrivent
$$\begin{cases}\bs_jx_i\bs_j\inv=x_i&\text{ si }i\neq j\\
         \bs_jy_i\bs_j\inv=y_i&\text{ si }i\neq j\\
         \bs_ix_i\bs_i\inv=x_{i-1}x_i\inv x_{i+1}&\text{ pour }i\neq 1\\
         \bs_iy_i\bs_i\inv=y_{i+1}y_i\inv y_{i-1}&\text{ pour }i\neq 1\\
         \bs_1x_1\bs_1\inv=y_2y_1\inv x_2&\\
         \bs_1y_1\bs_1\inv=y_2x_1\inv x_2&\\
\end{cases}
	 $$
Pour montrer que $U_\bI$ est le groupe libre engendr\'e par les $a_i$ et les
$b_i$ ou, ce qui revient au m\^eme, par les $x_i$ et les $y_i$
il suffit de montrer que les formules ci-dessus d\'efinissent bien une action
de $B_\bI$ sur un groupe libre. Plus pr\'ecis\'ement il suffit
de montrer que si $F$ est le groupe libre
engendr\'e par $\{\xi_1,\ldots,\xi_n,\eta_1,\ldots,\eta_{n-1}\}$,
les formules ci-dessus avec $x_i$ remplac\'e par $\xi_i$ et $y_i$ remplac\'e par
$\eta_i$ d\'efinissent une action de $B_\bI$ sur $F$.
Le groupe $F\rtimes B_\bI$ aura alors la m\^eme pr\'esentation que $D_\bI$ donc lui
sera isomorphe, l'isomorphisme envoyant $F$ sur $U_\bI$.

L'action d'un $\bs_j$ fix\'e donn\'ee par les
formules ci-dessus
est bien un automorphisme de $F$. Il suffit donc de voir que ces
automorphismes v\'erifient les relations de tresses.
Les relations de commutation
entre les actions
de $\bs_j$ et $\bs_k$ quand $|j-k|>1$ sont claires. Pour voir les
autres relations de tresses nous utilisons le lemme suivant qui r\'esulte d'un
simple calcul.
\begin{lemme}\label{auto du gpe libre}
Soient $\bs$ et $\bt$ deux automorphismes d'un groupe libre 
$G$ engendr\'e par quatre \'el\'ements $w$, $x$,
$y$ et $z$. On suppose que 
\begin{enumerate}
\item $\bs$ agit trivialement sur $x$, $w$ et $z$ et envoie $y$ sur
$xy\inv z$.
\item $\bt$ agit trivialement sur $y$, $w$ et $z$ et envoie $x$ sur
$wx\inv y$
\end{enumerate}
Alors $\bs\circ\bt\circ\bs=\bt\circ\bs\circ\bt$.
\end{lemme}
En appliquant une fois
le lemme avec $\bs=\bs_j$ et $\bt=\bs_{j+1}$, agissant
sur le
groupe libre engendr\'e par $w=\eta_{j+2}$, $x=\eta_{j+1}$, $y=\eta_j$ et
$z=\eta_{j-1}$
avec $j\geq 2$, et une deuxi\`eme fois avec les m\^emes \'el\'ements agissant sur
le groupe libre engendr\'e par
$w=\xi_{j-1}$, $x=\xi_j$, $y=\xi_{j+1}$ et $z=\xi_{j+2}$,
on voit que les actions
de $\bs_j$ et de $\bs_{j+1}$ v\'erifient la relation de tresses.
Il reste \`a voir la relation entre $\bs_1$ et $\bs_2$. On v\'erifie que
le compos\'e des automorphismes
$\bs_1\circ\bs_2\circ\bs_1\circ\bs_2$
comme le compos\'e
$\bs_2\circ\bs_1\circ\bs_2\circ\bs_1$
envoie
$\xi_1$ sur $\eta_3\eta_1\inv \xi_3$ (resp. $\xi_2$ sur 
$\eta_3\eta_2\inv \xi_3$, resp. $\eta_1$ sur $\eta_3\xi_1\inv \xi_3$,
resp. $\eta_2$ sur $\eta_3\xi_2\inv \xi_3$).
\end{proof}
\begin{remarque} Dans le cas $D_n$ avec $\bI$ comme ci-dessus, le groupe $U_\bI$
n'est \'evidemment pas libre \`a cause de la relation $a_2a_{2'}=a_{2'}a_2$,
mais ce n'est pas non plus le groupe engendr\'e par les $a_{\bb,\bs}$ avec cette
seule relation. Par exemple les \'el\'ements $a_{2'}\inv b_3 a_{2'}$ et $a_3$
commutent aussi (ce sont les images de $a_{2'}$ et $a_2$ par la conjugaison
par $\bs_2$).
\end{remarque}
Terminons cette section par une propri\'et\'e du d\'evissage \ref{devissage}
du groupe de tresses pur
dans les cas $A_n$, $B_n$ et $D_n$.
Si on prend $\bI$ comme ci-dessus, le groupe $B_\bI$ est un groupe de m\^eme
type, de rang $n-1$ (sauf si $i=2$ dans le cas $D_n$) et on peut d\'efinir une
suite de parties de $\bS$ embo{\^\i}t\'ees 
$\bI_0=\emptyset\subset\bI_1\ldots\subset\bI_{n-1}\subset\bI_n=\bS$ o\`u
$|\bI_i|=i$ sauf si $i=1$ pour le type $D_n$ auquel cas $\bI_1=\bI_0=\emptyset$.
Notons comme dans \ref{devissage} $U_j$ le groupe $U_\bI$ quand le
groupe ambient est $B_{\bI_j}$ et $\bI=\bI_{j-1}$.
\begin{proposition} \label{U_i-1 to U_i}
Avec les notations ci-dessus, pour $i\geq 2$,
la conjugaison par $\bs_i$ est un
isomorphisme de $U_{i-1}$ sur un sous-groupe de $U_i$ 
en tant que groupes munis d'une action de $B_{\bI_{i-2}}$.
\end{proposition}
\begin{proof} D'apr\`es les formules donnant les
g\'en\'erateurs de $U_{i-1}$, la conjugaison envoie bien $U_{i-1}$ dans $U_i$.
Comme $\bs_i$ centralise $B_{\bI_{i-2}}$, on obtient le r\'esultat.
\end{proof}
\section{Plongement du groupe de type $B_n$ dans le groupe de type $A_n$}
Dans cette section nous montrons que les actions des groupes de tresses de
type $A$ ou $B$ sur un groupe libre d\'efinie pr\'ec\'edemment sont fid\`eles et nous
en d\'eduisons un plongement du groupe de type $B_n$ dans le groupe de type
$A_n$.
\begin{proposition} \label{fidele} Les actions des groupes de tresses de
type $A$ ou $B$ sur des groupes libres d\'efinies plus haut sont
fid\`eles.
\end{proposition}
\begin{proof}
Ce r\'esultat est connu pour le type $A$. Nous le d\'emontrons pour le type $B$.
La d\'emonstration pour le type $A$ est analogue.
Remarquons d'abord que l'action du groupe de tresses de type $B_{n-1}$ sur le
groupe libre \`a $2n-1$ g\'en\'erateurs que nous avons d\'efinie induit sur
l'ab\'elianis\'e du groupe libre l'action de permutation du groupe de Coxeter de
type $B_{n-1}$ vu comme groupe de permutations de $2n-2$ \'el\'ements qui sont
ici les images
dans l'ab\'elianis\'e de $a_2,\ldots,a_n$ et de $b_2,\ldots,b_n$.
Un \'el\'ement du groupe de tresses
qui agit trivialement sur le groupe libre
est donc dans le groupe de tresses pur
$P_{n-1}$ de type $B_{n-1}$.
Nous avons d\'ecompos\'e le groupe de tresses pur $P_n$
en produit semi-direct $U_n\rtimes P_{n-1}$, o\`u $U_n$ est le groupe libre
engendr\'e par $a_1,\ldots,a_n$ et $b_2,\ldots,b_n$.
Chercher le noyau de l'action
revient donc \`a chercher les \'el\'ements de $P_{n-1}$ qui commutent \`a $U_n$ dans
$P_n$.

Consid\'erons un \'el\'ement $\bb\neq 1$ de $P_{n-1}$ qui agit trivialement sur
$U_n$ et soit $i$ le plus petit indice tel que $\bb$ soit dans
$P_i$. D'apr\`es \ref{U_i-1 to U_i} en utilisant la conjugaison par
$\bs_n\bs_{n-1}\ldots\bs_{i+2}$, on voit que $\bb$ commute \`a $U_{i+1}$,
autrement dit on est ramen\'e \`a $i=n-1$.

L'\'el\'ement $\bb$ s'\'ecrit de fa\c con unique $\bb_{n-1}\bb_{n-2}\ldots\bb_0$
avec $\bb_j\in U_j$ pour tout $j$ et $\bb_{n-1}\neq 1$.
Comme $\bb$ commute \`a $U_n$, il commute en particulier \`a $a_n$
Comme tout \'el\'ement de $B_{\bI_j}$ commute \`a $a_n=s_n^2$
pour $j=1,\ldots,n-2$ on voit que $\bb$ comute \`a
$a_n$ si et seulement si $\bb_{n-1}$ commute \`a $a_n$.
Conjuguons par $\bs_n$, on obtient un \'el\'ement $\lexp{\bs_n}\bb_{n-1}\in U_n$
qui commute \`a $a_n$. Or $U_n$ est un groupe libre dont $a_n$ est un des
g\'en\'erateurs. Donc $\lexp{\bs_n}\bb_{n-1}$ est une puissance de $a_n=\bs_n^2$.
Ceci implique que $\bb_{n-1}$ est une puissance de $a_n$,
ce qui est impossible puisque $\bb_{n-1}\not\in U_n$.
\end{proof}
Consid\'erons un groupe de tresses de type $A_n$ de g\'en\'erateurs
$\bs_1,\ldots,\bs_n$ et un groupe de tresses de type $B_n$ de g\'en\'erateurs
$\bs'_1,\ldots,\bs'_n$. On a $\bs_1^2\bs_2\bs_1^2\bs_2=\bs_2\bs_1^2\bs_2\bs_1^2$,
donc il existe un homomorphisme $\varphi$ du groupe de tresses
de type $B_n$ dans le groupe de tresses de type $A_n$ qui envoie 
$\bs'_1$ sur $\bs_1^2$ et $\bs_i'$ sur $\bs_i$ pour $i\geq 2$.
Notre but est de d\'emontrer le r\'esultat suivant.
\begin{theoreme}\label{plogement B_n->A_n}
L'homomorphisme $\varphi$ est injectif.
\end{theoreme}
Le th\'eor\`eme r\'esulte de la proposition suivante.
\begin{proposition} \label{plongement des groupes libres}
Consid\'erons un groupe libre $F$ \`a $n+1$ g\'en\'erateurs $a_1,\ldots,a_{n+1}$,
muni de
l'action du groupe de tresses de type $A_n$ d\'efinie plus haut et un groupe
libre $F'$ \`a $2n+1$ g\'en\'erateurs $a'_1,\ldots a'_{n+1}, b_2\ldots b_{n+1}$
muni de l'action du
groupe de tresses de type $B_n$; alors
\begin{enumerate}
\item L'homomorphisme $\psi$ de $F'$ dans $F$ d\'efini par
$$\begin{cases}\psi(a'_1)=a_1^2,&\\
\psi(a'_i) =a_i,&pour $i=2\ldots n+1$\\
\psi(b_i)=a_1a_2\ldots a_ia_{i-1}\inv\ldots a_2\inv a_1\inv&
\text{ pour }i=2\ldots n+1\\
\end{cases}$$
est injectif.
\item L'image $\psi(F')$ est un sous-groupe normal d'indice 2 de $F$.
\item L'action du groupe de tresses de type
$B_n$ sur $F$ \`a travers $\varphi$ stabilise $\psi(F')$ et $\psi$ est un
morphisme de groupes munis de l'action du groupe de tresses de type $B_n$.
\end{enumerate}
\end{proposition}
Le th\'eor\`eme se d\'eduit imm\'ediatement de la proposition car
par (iii) et (i) un \'el\'ement de $\ker\varphi$ agit trivialement 
sur $F'$ et l'action sur $F'$ est fid\`ele par \ref{fidele}.
\begin{proof}[D\'emonstration de \ref{plongement des groupes libres}]
Le (iii) r\'esulte du calcul. Pour prouver (i) et (ii) nous changeons de
g\'en\'erateurs dans $F$. On pose $x_i=a_1\ldots a_i$. 
Le groupe $F$ est aussi le groupe libre engendr\'e par
$x_1,\ldots,x_{n+1}$. On a
$$\begin{cases}\psi(a'_1)=x_1^2,&\\
\psi(a'_i) = x_{i-1}\inv x_i,&\text{ pour }i=2\ldots n+1\\
\psi(b_i)= x_i x_{i-1}\inv& pour $i=2\ldots n+1$.\\
\end{cases}$$
Donc $\psi(F')$ contient tous les produits $x_ix_j$, $x_i\inv x_j$ et $x_i
x_j\inv$ avec $i,j\in\{1,\ldots,n+1\}$. C'est donc le sous-groupe de $F$ form\'e
des \'el\'ements de longueur paire en les $x_i$ et leurs inverses.
Ce sous-groupe est d'indice 2 dans $F$.
Pour montrer (i) nous appliquons la m\'ethode de Reidemeister-Schreier pour
obtenir une pr\'esentation de $\psi(F')$. Un ensemble de repr\'esentants des
classes est form\'e de $\{1,x_1\}$. On obtient
le syst\`eme de g\'en\'erateurs fourni par la
m\'ethode en multipliant $x_i$ et $x_1x_i$ par les inverses de leurs
repr\'esentants respectifs qui sont $x_1$ et $1$. On obtient donc $x_ix_1\inv$
pour $i=2,\ldots ,n+1$ et $x_1x_i$ pour $i=1,\ldots,n+1$.
La m\'ethode dit qu'il n'y a aucune relation entre ces g\'en\'erateurs. Donc
$\psi(F')$ est le groupe libre engendr\'e par ces \'el\'ements, donc aussi le groupe
libre engendr\'e par
$x_1^2$, les
$x_{i-1}\inv x_i$ et les
$x_i x_{i-1}\inv$ pour $i=2\ldots n+1$, car le passage d'un syst\`eme de
g\'en\'erateurs \`a l'autre est inversible.

Comme $\psi(F')$ est libre engendr\'e par les images (distinctes)
des g\'en\'erateurs de $F'$
l'homomorphisme $\psi$ est injectif.
\end{proof}
\begin{remarque} On peut aussi d\'emontrer ce r\'esultat topologiquement en
interpr\'etant le groupe de tresses de type $B_n$ comme une partie des tresses
classiques \`a $n+1$ brins. Nous avons voulu en donner une preuve purement
combinatoire.
\end{remarque}
\section{Une extension du groupe de Coxeter}
Consid\'erons le groupe quotient $B_W/D(P_W)$. D'apr\`es \ref{noyau de N}
on a une suite exacte
$$1\to P_W/D(P_W)\simeq (2\bbZ)^T\to B_W/D(P_W)\to W\to 1.$$
Le groupe $W$ agit naturellement sur $(2\bbZ)^T$ et se rel\`eve canoniquement
en $\Bred$ dans $B_W$ donc aussi dans $B_W/D(P_W)$.
L'extension donn\'ee par la suite
exacte ci-dessus est donc associ\'ee \`a un cocycle
que nous allons calculer. On sait que l'isomorphisme $P_W/D(P_W)\simeq
(2\bbZ)^T$ est compatible \`a l'action de $B_W$ sur $(2\bbZ)^T$ (via $W$).
Soient $\bv$ et $\bw$ les rel\`evements dans $\Bred$ de deux \'el\'ements
$v$ et $w$ de $W$; nous les consid\'erons comme des \'el\'ements de
$B_W/D(P_W)$;
si $a$ et $b$ sont dans $(2\bbZ)^T\subset B_W/D(P_W)$,
dans $B_W/D(P_W)$ on a
$a\bv b\bw=a\lexp v b\bv\bw=a\lexp v
b(\bv\bw\underline{vw}\inv)\underline{vw}$ o\`u $\underline{vw}$ est le relev\'e
dans $\Bred$ de $vw$.
L'\'el\'ement $(\bv\bw\underline{vw}\inv)$ est dans $P_W/D(P_W)$
et, vu comme \'el\'ement de $(2\bbZ)^T$, est \'egal \`a
$N(\bv\bw\underline{vw}\inv)$. De plus on a
$N(\bv\bw\underline{vw}\inv)=N(\bv\bw)-N(\underline{vw})$. Donc le cocycle
qui d\'efinit l'extension est $(v,w)\mapsto N(\bv\bw)-N(\underline{vw})=
N(\bv)+\lexp vN(\bw)-N(\underline{vw})$.
\begin{proposition} \label{extension de W}L'extension $B_W/D(P_W)$ de $W$
par $(2\bbZ)^T$ n'est pas un produit semi-direct.
\end{proposition}
\begin{proof}
Il faut montrer que le cocycle n'est pas \'equivalent au cocycle trivial. 
Ceci \'equivaut \`a montrer qu'il n'existe pas d'application $\alpha:W\to
(2\bbZ)^T$ telle que $N(\bv)+\lexp vN(\bw)-N(\underline{vw})=
\alpha(v)+\lexp v\alpha(w)-\alpha(vw)$.
Si $\alpha$ existe, posons $f(w)=N(\bw)-\alpha(w)$.
On a $f(vw)=f(v)+\lexp vf(w)$ pour
tous $v$ et $w$ dans $W$. Ceci implique $f(1)=0$ et $\lexp vf(v)=-f(v)$
pour tout $v$ d'ordre 2. En particulier pour $v=s\in S$. Donc le coefficient
de $f(s)$ sur $s$ est nul, ce qui implique que le coefficient de $\alpha(s)$ sur
$s$ est \'egal \`a celui de $N(\bs)$, c'est-\`a-dire $1$.
Ceci est contradictoire avec
le fait que tous les coefficients de $\alpha$ sont pairs.
\end{proof}

\end{document}